\documentclass[a4paper]{amsart}
\usepackage{dsfont}
\usepackage{amsmath,amsthm,amssymb,hyperref,mathrsfs}
\usepackage{aliascnt}
\usepackage{lmodern}
\usepackage[T1]{fontenc}
\usepackage[textsize=footnotesize]{todonotes}

\usepackage{xcolor}
\definecolor{dblue}{rgb}{0,0,0.70}
\hypersetup{
	unicode=true,
	colorlinks=true,
	citecolor=dblue,
	linkcolor=dblue,
	anchorcolor=dblue
}

\makeatletter
\expandafter\g@addto@macro\csname th@plain\endcsname{%
		\thm@notefont{\bfseries}
	}%
\expandafter\g@addto@macro\csname th@remark\endcsname{%
		\thm@headfont{\bfseries}
	}%
\makeatother

\newtheorem{theorem}
{Theorem}[section]
\newtheorem*{theorem*}{Theorem}

\newaliascnt{lemma}{theorem}

\aliascntresetthe{lemma}
\newtheorem*{lemma*}{Lemma}

\newtheorem{claim}[theorem]{Claim}

\newaliascnt{fact}{theorem}

\aliascntresetthe{fact}

\newaliascnt{proposition}{theorem}
\newtheorem{proposition}[proposition]{Proposition}
\aliascntresetthe{proposition}
\newtheorem*{proposition*}{Proposition}

\newaliascnt{corollary}{theorem}
\newtheorem{corollary}[corollary]{Corollary}
\aliascntresetthe{corollary}

\theoremstyle{remark}

\newaliascnt{remark}{theorem}

\aliascntresetthe{remark}
\newaliascnt{question}{theorem}
\newtheorem{question}[question]{Question}
\aliascntresetthe{question}
\newaliascnt{conjecture}{theorem}

\aliascntresetthe{conjecture}

\newtheorem*{question*}{Question}

\newaliascnt{definition}{theorem}
\newtheorem{definition}[definition]{Definition}
\aliascntresetthe{definition}

\newaliascnt{example}{theorem}
\newtheorem{example}[example]{Example}
\aliascntresetthe{example}

\renewcommand{\restriction}{\mathbin\upharpoonright}

\newcommand{\axiom}[1]{\mathsf{#1}}
\newcommand{\ZFC}{\axiom{ZFC}}
\newcommand{\AC}{\axiom{AC}}

\newcommand{\DC}{\axiom{DC}}
\newcommand{\ZF}{\axiom{ZF}}
\newcommand{\ZFA}{\axiom{ZFA}}

\newcommand{\COrd}{\mathrm{COrd}}

\newcommand{\GCH}{\axiom{GCH}}

\newcommand{\HS}{\axiom{HS}}

\DeclareMathOperator{\cf}{cf}
\DeclareMathOperator{\dom}{dom}

\DeclareMathOperator{\supp}{supp}

\DeclareMathOperator{\sym}{sym}

\DeclareMathOperator{\fix}{fix}

\DeclareMathOperator{\id}{id}

\DeclareMathOperator{\Add}{Add}

\DeclareMathOperator{\tcl}{tcl}

\newcommand{\forces}{\mathrel{\Vdash}}

\newcommand{\PP}{\mathbb P}
\newcommand{\QQ}{\mathbb Q}

\newcommand{\cD}{\mathcal D}

\newcommand{\cM}{\mathcal M}

\newcommand{\cS}{\mathcal S}

\newcommand{\sF}{\mathscr F}
\newcommand{\sG}{\mathscr G}

\newcommand{\1}{\mathds 1}

\newcommand{\tup}[1]{\langle#1\rangle}

\author{Asaf Karagila}
\author{Jonathan Schilhan}
\email{karagila@math.huji.ac.il}
\urladdr{http://karagila.org}
\email{j.schilhan@leeds.ac.uk}
\urladdr{http://www.logic.univie.ac.at/~schilhanj/}
\address{School of Mathematics,
    University of Leeds.
    Leeds, LS2~9JT, UK}
\thanks{The authors were supported by a UKRI Future Leaders Fellowship [MR/T021705/1].}

\date{December 21, 2022}

\title[Sequential and distributive forcings without choice]{Sequential and distributive forcings\\ without choice}

\begin{document}
\begin{abstract}
  In the Zermelo--Fraenkel set theory with the Axiom of Choice a forcing notion is ``$\kappa$-distributive'' if and only if it is ``$\kappa$-sequential''. We show that without the Axiom of Choice this equivalence fails, even if we include a weak form of the Axiom of Choice, the Principle of Dependent Choice for $\kappa$. Still, the equivalence may still hold along with very strong failures of the Axiom of Choice, assuming the consistency of large cardinal axioms. We also prove that while a $\kappa$-distributive forcing notion may violate Dependent Choice, it must preserve the Axiom of Choice for families of size $\kappa$. On the other hand, a $\kappa$-sequential can violate the Axiom of Choice for countable families. We also provide a condition of ``quasiproperness'' which is sufficient for the preservation of Dependent Choice, and is also necessary if the forcing notion is sequential.
\end{abstract}
\maketitle              
\section{Introduction}
The method of forcing was developed by Paul Cohen in 1963 to prove that the Continuum Hypothesis cannot be proved from the Zermelo--Fraenkel set theory with the Axiom of Choice ($\ZFC$). The technique works by picking a partial order approximating a ``generic set'' that can be added to a ``ground model'' of set theory while preserving the axioms of $\ZFC$. We understand the general theory of forcing fairly well when working in $\ZFC$. For example, if we chose a partial order which is countably distributive, then the generic extension of the universe will not have any new countable sequences of ground model elements. This property implies, among other things, that no new real numbers are added, and that $\omega_1$, the least uncountable cardinal, is the same between the ground model and its generic extension. On the other hand, we know that distributivity assumptions are not enough to prove that stationary subsets of $\omega_1$ remain stationary.\footnote{We can think of stationary sets as similar to sets which are not null in comparison to the Lebesgue measure on the unit interval.}

While the basic machinery of forcing does not rely on the Axiom of Choice, its general theory makes heavy use of it. This means that working over general models of Zermelo--Fraenkel ($\ZF$), where the Axiom of Choice is not necessarily assumed, is significantly harder: our intuition was honed in $\ZFC$ for many decades, and we still do not have a complete picture of what could go wrong, or how do our standard definitions behave in general models of $\ZF$. With the recent advents of very large cardinal axioms,\footnote{These are axioms that go beyond $\ZFC$, the most famous one is perhaps ``there is an inaccessible cardinal'', or equivalently ``there is a Grothendieck universe''.} e.g.\ Reinhardt and Berkeley cardinals whose existence refutes the Axiom of Choice, it is very important to better understand the theory of forcing in $\ZF$.

In this paper we separate two properties which are equivalent in $\ZFC$, namely, distributivity and adding new sequences of ground model objects, which we term ``sequential''. Our main result is that this equivalence is not provable from $\ZF$, or even $\ZF$ augmented by the Principle of Dependent Choice ($\DC$) and its generalised versions. Moreover, we show that forcing with a distributive partial order must preserve the Axiom of Choice for countable families of sets ($\AC_\omega$), but can violate $\DC$, whereas a sequential partial order may even violate $\AC_\omega$ itself. We also provide a necessary and sufficient condition for a sequential partial order to preserve $\DC$, termed here ``quasiproperness''.

Finally, we provide a partial answer to the question of whether or not the equivalence between the two properties is itself equivalent to the Axiom of Choice. We prove that in the Gitik model, where all the limit ordinals have countable cofinality, the equivalence between the two properties holds, while the Axiom of Choice fails quite badly. The one drawback is that the Gitik model requires assuming the consistency of suitable large cardinal axioms, which leaves the question of whether or not the equivalence can hold in the absence of the Axiom of Choice without these additional assumptions wide open.

\subsection{In this paper}
We begin by covering the basics of symmetric extensions, our main technical tool for constructions models of $\ZF$. In \autoref{sec:dist} we study the basic properties of distributive and sequential forcings. \autoref{sec:minor} is dedicated for two minor results in the study of preservation of choice principles under generic extensions, we define a property akin to properness and show that it is equivalent to the preservation of $\DC$, at least for sequential forcings. \autoref{sec:major} is dedicated for our main theorem. Finally, \autoref{sec:questions} concludes the paper with several open questions that arise from this work.
\subsection*{Acknowledgements}
The authors would like to thank Jonathan Kirby and Mark Kamsma for their comments regarding the introduction of this paper. We would also like to thank the anonymous referee for their helpful remarks.
\section{Preliminaries}
Throughout this paper we work in $\ZF$, unless specified otherwise. Our treatment of forcing will be standard. If $\PP$ is a notion of forcing, then $\PP$ is a preordered set with a maximum element denoted by $\1_\PP$, or with the subscript omitted when clear from context. We write $q\leq p$ to mean that $q$ is a \textit{stronger} condition than $p$, or that it \textit{extends} $p$. Two conditions are compatible if they have a common extension. We will also follow Goldstern's alphabet convention so $p$ is never a stronger condition than $q$, etc.

When given a collection of $\PP$-names, $\{\dot x_i\mid i\in I\}$, we will denote by $\{\dot x_i\mid i\in I\}^\bullet$ the canonical name this class generates: $\{\tup{\1,\dot x_i}\mid i\in I\}$. This notation extends naturally to ordered pairs and functions whose domain is in the ground model. We will also say that $\dot y$ \textit{appears} in $\dot x$ if there is some $p\in\PP$ such that $\tup{p,\dot y}\in\dot x$.

Given a set $X$, we use $|X|$ to denote its cardinal number. If $X$ can be well-ordered, then $|X|$ is simply the least ordinal equipotent with $X$. Otherwise, we use the Scott cardinal of $X$ which is the set $\{Y\in V_\alpha\mid\exists f\colon X\to Y\text{ a bijection}\}$ with $\alpha$ taken as the least ordinal for which the set is non-empty. Greek letters, when used as cardinals, will always refer to well-ordered cardinals. We will denote by $\COrd$ the class of well-orderable cardinals, that is the finite ordinals and the $\aleph$ numbers.

We write $|X|\leq|Y|$ to mean that there is an injection from $X$ into $Y$, and we write $|X|<|Y|$ to mean that there is an injection, but there is no injection from $Y$ into $X$. Note that unlike in the case of $\ZFC$, writing $|X|\nleq|Y|$ does not imply that $|Y|<|X|$.

We write $|X|\leq^*|Y|$ to mean that there is a surjection from a subset of $Y$ onto $X$.\footnote{We have no need for the case $<^*$ here, but to dispel any ambiguity, $|X|<^*|Y|$ means that there is a surjection from $Y$ onto $X$, but no surjection from $X$ onto $Y$, which is stronger than saying $|X|\leq^*|Y|$ and $|X|\neq|Y|$.} This relation is transitive, not necessarily antisymmetric (unlike $\leq$).

The axiom $\AC_X$ states that given any family of non-empty set indexed by $X$ admits a choice function, we omit $X$ to mean $\forall X\,\AC_X$. For an infinite cardinal $\kappa$, the axiom $\DC_\kappa$ states that every $\kappa$-closed tree\footnote{Recall that a tree is $\kappa$-closed if for all $\alpha<\kappa$, every chain of order type $\alpha$ has an upper bound.} has a maximal element or a chain of order type $\kappa$. We write $\DC_{<\kappa}$ to mean $(\forall\lambda<\kappa)\DC_\lambda$. In the case of $\DC_\omega$ we simply write $\DC$.

\subsection{Symmetric extensions}
Forcing is an extremely versatile technique when it comes to independence proofs. It has one drawback: a generic extension of a model of $\ZFC$ is a model of $\ZFC$.\footnote{You could say that this is not a bug, but a feature, and you would not be wrong. But it is a problem when we want to prove independence results related to the axiom of choice.} But we can extend the technique of forcing. By imitating the Fraenkel--Mostowski--Specker technique for permutation models\footnote{In the context of $\ZFA$, that is $\ZF$ with atoms.} we can identify a class of names which defines an intermediate model, between the ground model and its generic extension, where the axiom of choice may fail.

Let $\PP$ be a fixed forcing notion. If $\pi$ is an automorphism of $\PP$, then $\pi$ extends to $\PP$-names by recursion: \[\pi\dot x=\{\tup{\pi p,\pi\dot y}\mid\tup{p,\dot y}\in\dot x\}.\]
Seeing how the forcing relation is defined from the order, the following lemma is not surprising. For a proof of this lemma, see Lemma~14.37 in \cite{Jech:ST2003}.
\begin{lemma*}[The Symmetry Lemma]
Let $\PP$ be a forcing, $\pi$ an automorphism of $\PP$, $p\in\PP$, and $\dot x$ some $\PP$-name. Then \[p\forces\varphi(\dot x)\iff\pi p\forces\varphi(\pi\dot x).\qed\]
\end{lemma*}

  Let $\sG$ be a group, we say that $\sF$ is a \textit{filter of subgroups} if it is a non-empty collection of subgroups of $\sG$ which is closed under supergroups and finite intersections. We say that $\sF$ is \textit{normal} if whenever $H\in\sG$ and $\pi\in\sG$, then $\pi H\pi^{-1}\in\sF$ as well.

We say that $\tup{\PP,\sG,\sF}$ is a \textit{symmetric system} if $\PP$ is a forcing notion, $\sG$ is a group of automorphisms of $\PP$, and $\sF$ is a normal filter of subgroups of on $\sG$. Given such symmetric system, we say that a $\PP$-name, $\dot x$, is \textit{$\sF$-symmetric} if $\sym_\sG(\dot x)=\{\pi\in\sG\mid\pi\dot x=\dot x\}\in\sF$. We say that $\dot x$ is hereditarily $\sF$-symmetric, if this notion holds for every $\PP$-name hereditarily appearing in $\dot x$. We denote by $\HS_\sF$ the class of hereditarily $\sF$-symmetric names.
\begin{theorem*}
Let $\tup{\PP,\sG,\sF}$ be a symmetric system, $G\subseteq\PP$ a $V$-generic filter, and let $M$ denote the class $\HS_\sF^G=\{\dot x^G\mid\dot x\in\HS_\sF\}$. Then $M$ is a transitive model of $\ZF$ satisfying $V\subseteq M\subseteq V[G]$.
\end{theorem*}
We say that $M$ as in the theorem above, whose proof appears as Lemma~15.51 in \cite{Jech:ST2003}, is a \textit{symmetric extension} of $V$. The symmetric extensions of $V$ were studied recently by Usuba in \cite{Usuba:LS,Usuba:GeologySym}. It is tempting to think that every intermediate model of $\ZF$ is a symmetric extension, but this is not true, as was shown in \cite{Karagila:Bristol}.

Since we will only have a single symmetric extension of concern at each step, even if we will force over it, we will omit the subscripts from the notation from here on out.

Finally, we have a forcing relation for symmetric extensions, $\forces^\HS$ defined by relativising the $\forces$ relation to the class $\HS$. This relation has the same properties and behaviour as the standard $\forces$ relation. Moreover, if $\pi\in\sG$, then the Symmetry Lemma holds also for $\forces^\HS$.

We conclude this introduction with a general example.

\begin{example}\label{ex:cohen}
  Let $\kappa$ and $\lambda$ be regular cardinals such that $\lambda\geq\kappa$ and suppose that $\kappa^{<\kappa}=\kappa$. Let $\PP$ be the forcing $\Add(\kappa,\lambda)$, whose conditions are partial functions $p\colon\lambda\times\kappa\to 2$ such that $|p|<\kappa$, the projection of $p$ onto its $\lambda$ component is called the \textit{support} of $p$ and is denoted by $\supp p$. We let $\sG$ be the group of permutations of $\lambda$, and $\pi\in\sG$ acts on $\PP$ by letting \[\pi p(\pi\alpha,\beta)=p(\alpha,\beta).\] Finally, let the filter of subgroups be generated by $\{\fix(E)\mid E\in[\lambda]^{<\lambda}\}$, where $\fix(E)=\{\pi\in\sG\mid\pi\restriction E=\id\}$.

  We denote by $\dot a_\alpha$, for $\alpha<\lambda$, the name of the $\alpha$th generic subset: \[\{\tup{p,\check\beta}\mid p(\alpha,\beta)=1\}.\] We will denote by $\dot A$ the name $\{\dot a_\alpha\mid\alpha<\lambda\}^\bullet$. Let $G$ be a $V$-generic filter and let $M$ be the corresponding symmetric extension, we will omit the dots to indicate the interpretation of the names in $M$. We will show that the following hold in $M$:
  \begin{enumerate}
  \item Every well-orderable subset of $A$ has size $<\lambda$.
  \item $\DC_{<\lambda}+\lnot\AC$.
  \end{enumerate}

  First we observe that $\pi\dot a_\alpha=\dot a_{\pi\alpha}$, and since all the names appearing in $\dot a_\alpha$ are canonical ground model names, $\fix(\{\alpha\})$ witnesses that $\dot a_\alpha\in\HS$. Consequently, $\pi\dot A=\dot A$ for all $\pi\in\sG$. Therefore, each $a_\alpha$ and $A$ itself are all in $M$.

  Suppose that $\dot B\in\HS$ and $p\forces^\HS``\dot B\subseteq\dot A$ and can be well-ordered.'' Let $E\subseteq\lambda$ be such that $\fix(E)\subseteq\sym(\dot f)$, where $p\forces^\HS``\dot f\colon\dot B\to\check\eta$ is an injective function'', and we may also assume that $\supp p\subseteq E$. Note that $\fix(E)\subseteq\sym(\dot B)$ as well.

  Let $\alpha<\lambda$ be such that $\alpha\notin E$, let $q\leq p$ be a condition such that $q\forces^\HS\dot a_\alpha\in\dot B$ and without loss of generality we also assume that for some $\delta<\eta$, $q\forces^\HS\dot f(\dot a_\alpha)=\check\delta$.

  Since $\supp q$ is of size $<\kappa$, we can find $\beta\notin\supp q\cup E$ and consider $\pi$ to be the automorphism defined by the $2$-cycle $(\alpha\ \beta)$. By the choice of $\alpha$ we immediately have that $\pi\in\fix(E)$ and therefore $\pi p=p,\pi\dot B=\dot B,\pi\dot f=\dot f$. Applying these, along with the Symmetry Lemma, we get that $\pi q\forces^\HS\dot f(\dot a_\alpha)=\check\delta$. But $\pi q$ is compatible with $q$, as we only moved one coordinate to a previously-empty one. This means that $r=q\cup\pi q$ is a condition which forces both ``$\dot f$ is injective'' and $\dot f(\dot a_\alpha)=\dot f(\dot a_\beta)$. This is of course impossible. This means that there is no such $q$, and therefore if $\alpha\notin E$, $p\forces^\HS\dot a_\alpha\notin\dot B$. Since $E\in[\lambda]^{<\lambda}$, and since $\lambda$ was not collapsed in $V[G]$, the conclusion holds.

  This immediately shows that $\lnot\AC$ holds as well. To get $\DC_{<\lambda}$ we appeal to \cite{Karagila:DC}, where the folklore results about preservation of Dependent Choice principles are formalised. Specifically, if $\PP$ is $\lambda$-closed or has $\lambda$-c.c., which in this case follows from the assumption $\kappa^{<\kappa}=\kappa$, and $\sF$ is $\lambda$-complete, then $\DC_{<\lambda}$ holds in the symmetric extension.
\end{example}

In the case $\lambda=\kappa$ we refer to this model as the $\kappa$-Cohen model, and if $\kappa=\omega$ we omit it altogether. The Cohen model is one of the most important models of $\ZF+\lnot\AC$. It satisfies the Boolean Prime Ideal theorem, and has many interesting properties. For a complete exposition, see Chapter~5 in \cite{Jech:AC}.
\section{Distributive and sequential forcings}\label{sec:dist}
\begin{definition}
We say that a forcing notion $\PP$ is \textit{${\leq}|X|$-distributive} if whenever $\tup{D_x\mid x\in X}$ is a family of dense open sets, $\bigcap_{x\in X}D_x$ is a dense open set.\footnote{The intersection is always open, in the case of forcing, so we really only care about its density.} If $X$ can be well-ordered, we will use the standard notation of $\kappa$-distributive to mean ``for all $\lambda<\kappa$, ${\leq}\lambda$-distributive'', and we will use $\sigma$-distributive to mean $\aleph_1$-distributive.
\end{definition}
To make the definition smoother, we consider the intersection as bounded by $\PP$, namely $\bigcap_{x\in X}D_x=\{p\in\PP\mid\forall x\in X, p\in D_x\}$. This has the benefit that for $X=\varnothing$, $\bigcap_{x\in X}D_x=\PP$.
\begin{proposition}\label{prop:*-dist}
If $\PP$ is ${\leq}|X|$-distributive and $|Y|\leq^*|X|$, then $\PP$ is also ${\leq}|Y|$-distributive.
\end{proposition}
\begin{proof}
Let $f\colon X\to Y$ be a surjective function (for $Y=\varnothing$ the conclusion is vacuously true), if $\tup{D_y\mid y\in Y}$ is a family of dense open sets, let $E_x=D_{f(x)}$, then $\tup{E_x\mid x\in X}$ is a family of dense open sets indexed by $X$ and therefore its intersection is dense. Easily, $\bigcap_{y\in Y}D_y=\bigcap_{x\in X}E_x$, and so $\bigcap_{y\in Y}D_y$ is dense.
\end{proof}
\begin{definition}
  Let $\cD$ be a class (possibly proper) of cardinals. We define the following properties.
  \begin{enumerate}
  \item $\cD$ is \textit{$*$-closed} if whenever $|X|\in\cD$ and $|Y|\leq^*|X|$, then $|Y|\in\cD$.\footnote{One must resist the knee-jerk reaction to use the term ``projective'' here as that word is used too often.}
  \item $\cD$ is \textit{union-regular} if it is directed and whenever $|X|\in\cD$ and $\tup{A_x\mid x\in X}$ is a sequence of sets such that $|A_x|\in\cD$ for all $x\in X$, then there is some $|A|\in\cD$ such that for all $x$, $|A_x|\leq|A|$.
  \end{enumerate}
\end{definition}
\begin{theorem}\label{thm:dist-spec}
Let $\PP$ be a forcing and let $\cD$ be the class of cardinals such that $|X|\in\cD$ if and only if $\PP$ is ${\leq}|X|$-distributive. Then $\omega\subseteq\cD$ and $\cD$ is $*$-closed and union-regular.
\end{theorem}
\begin{proof}
  The fact that $\omega\subseteq\cD$ is trivial. The fact that it is $*$-closed follows from \autoref{prop:*-dist}. Finally, suppose that $|X|\in\cD$ and for each $x\in X$, $A_x$ is some set such that $|A_x|\in\cD$, without loss of generality assume that $A_x$ are disjoint, since the disjoint union maps onto the union in the obvious way and $\cD$ is $*$-closed.

  Let $A=\bigcup_{x\in X} A_x$ and suppose that $\tup{D_a\mid a\in A}$ is a family of dense open subsets of $\PP$. For each $x\in X$, consider $\tup{D_a\mid a\in A_x}$, then due to the fact that $|A_x|\in\cD$, we can replace $\tup{D_a\mid a\in A_x}$ by its intersection, $E_x$. This means that $\bigcap_{a\in A}D_a$ is the same as $\bigcap_{x\in X}E_x$, but since $|X|\in\cD$ as well, the intersection is dense as wanted.
\end{proof}
We will refer to $\cD$ in the theorem as the \textit{distributivity spectrum of $\PP$} and denote it by $\cD_\PP$.
\begin{corollary}\label{cor:gitik}
Suppose that $\cf(\alpha)=\omega$ for any limit ordinal $\alpha$. If $\PP$ is $\sigma$-distributive, then $\COrd\subseteq\cD_\PP$. Moreover, suppose that every set in $V$ is generated by iterating countable unions starting with the class $[V]^{\leq\omega}$, then any $\sigma$-distributive forcing is trivial.\footnote{Note that if $\PP$ is ${\leq}|\PP|$-distributive, then it must be trivial.}\qed
\end{corollary}
The conditions above seem fantastic, especially the latter, but they are indeed consistent with $\ZF$,\footnote{Assuming the consistency of suitable large cardinal axioms.} as shown by Gitik in \cite[Theorem~6.3]{Gitik:1980}.

\begin{definition}
  We say that a forcing $\PP$ is \textit{${\leq}|X|$-sequential} if whenever $G\subseteq\PP$ is $V$-generic and $f\in V[G]$ is a function $f\colon X\to V$, then $f\in V$. The same caveats regarding well-ordered $X$ will apply here as they do for distributivity.
\end{definition}
\begin{proposition}
  Suppose that $\PP$ is ${\leq}|X|$-distributive, then it is ${\leq}|X|$-sequential.
\end{proposition}
\begin{proof}
  Suppose that $\dot f$ is a $\PP$-name such that $\1\forces\dot f\colon\check X\to\check V_\alpha$, for some $\alpha$, defining $D_x=\{p\in\PP\mid\exists y(p\forces\dot f(\check x)=\check y\}$, we have that $D_x$ is a dense open set. By distributivity, $D=\bigcap_{x\in X}D_x$ is dense, and if $p\in D$, we define $f_p(x)=y$ if and only if $p\forces\dot f(\check x)=\check y$. Since $p\in D_x$ for all $x\in X$, this function is well-defined, and easily $p\forces\dot f=\check f_p$.
\end{proof}
It is a standard exercise that assuming $\ZFC$, ${\leq}|X|$-sequential also implies ${\leq}|X|$-distributive. The proof, however, relies on the fact that every dense open set contains a maximal antichain (which makes the generic filter act as a choice function). As we will see in \autoref{sec:major}, this reliance on the axiom of choice is crucial.

Nevertheless, defining the sequentiality spectrum of a forcing $\PP$, denoted by $\cS_\PP$, in an analogous manner to the distributivity spectrum, the proofs of \autoref{thm:dist-spec} and \autoref{cor:gitik} work also for the sequentiality spectrum.

\begin{corollary}
It is consistent with $\ZF+\lnot\AC$ that for every $X$, every ${\leq}|X|$-sequential forcing is ${\leq}|X|$-distributive.
\end{corollary}
\begin{proof}
  First we will show that in Gitik's model every $\sigma$-sequential forcing is trivial. Define a rank function in the following way: $[V]^{\leq\omega}$ are the sets of rank $0$, the successor steps are countable union of sets from previous ranks, and the limit steps are unions of previous ranks. As we remarked, in Gitik's model, every set has a rank in that sense.

  By induction on this rank, if $A$ is a least ranked set which has a fresh subset, $B$, in a generic extension, let $\{A_n\mid n<\omega\}$ be a countable sequence of sets of lower rank whose union is $A$, then either $\{A_n\cap B\mid n<\omega\}$ is a fresh sequence, or $B\cap A_n$ is fresh for some $n<\omega$. Since $A$ is minimally ranked, the latter is impossible, and so the generic extension must not be $\sigma$-sequential.

  Next, since every infinite set is a countable union of sets of smaller cardinality, every infinite set can be mapped onto $\omega$. So, by $*$-closure of $\cS_\PP$, if $X$ is infinite and $\PP$ is ${\leq}|X|$-sequential, then $\PP$ is $\sigma$-sequential, and thus trivial.
\end{proof}

\section{Some minor positive results about distributive forcings}\label{sec:minor}
\begin{theorem}\label{thm:dist-pres}
Suppose that $\AC_X$ holds, if $\PP$ is ${\leq}|X|$-distributive, then $\AC_X$ is preserved.
\end{theorem}
\begin{proof}
  Suppose that $\dot F$ is a $\PP$-name and $\1\forces``\dot F$ is a function with domain $\check X$ and $\dot F(\check x)\neq\check\varnothing$ for all $x\in X$''. For each $x\in X$ let $D_x$ be the dense open set $\{p\in\PP\mid\exists\dot y(p\forces\dot y\in\dot F(\check x))\}$. Suppose that $p\in\bigcap_{x\in X}D_x$, then for all $x\in X$ the class $\{\dot y\in V^\PP\mid p\forces\dot y\in\dot F(\check x)\}$ is non-empty. Using Scott's trick, we may assume that each of these is a set. Applying $\AC_X$ in $V$, there is a function $f$ such that for all $x\in X$, $f(x)=\dot y$ and $p\forces\dot y\in\dot F(\check x)$. This lets us define an obvious name for a choice function below $p$.

  Since $\bigcap_{x\in X}D_x$ is dense, $\1\forces\exists f\forall x\in\check X(f(x)\in\dot F(x))$ as wanted.
\end{proof}
\begin{proposition}\label{prop:gnd-model-ac}
If $\PP$ is ${\leq}|X|$-sequential and $\forces_\PP\AC_X$, then $\AC_X$ holds in $V$.\qed
\end{proposition}

If we concentrate on the case where $X=\omega$, this shows that a $\sigma$-distributive must preserve $\AC_\omega$. From the work of the first author with David Asper\'o in \cite{AsperoKaragila:2020} we know that a proper forcing, and in particular a $\sigma$-closed forcing,\footnote{Recall that $\PP$ is $\sigma$-closed if every countable sequence of decreasing conditions has a lower bound. The statement ``Every $\sigma$-closed forcing is $\sigma$-distributive'' is equivalent to $\DC$ (see \cite{Karagila:DC}).} must preserve $\DC$, so the natural question now is: does $\sigma$-distributive suffice for the proof?

There are reasons to expect a positive answer. For example, assuming $\AC$ holds, if $\tup{\PP,\sG,\sF}$ is a symmetric system where $\PP$ is $\sigma$-distributive and $\sF$ is $\sigma$-complete, then $\DC$ holds in the symmetric extension (see \cite{Banerjee:2020} and \cite{KaragilaSchilhan:Bristol}, for example). As the main theorem of this paper shows, however, this is not the case in $\ZF$. One is left asking, is there a property between $\sigma$-distributive and proper which preserves $\DC$?

\begin{definition}
  We say that a forcing $\PP$ is \textit{quasiproper} if for every $p\in\PP$ and $\PP$-name $\dot X$ there is a countable elementary submodel, $M$, of some large enough $H(\kappa)$ such that $p,\PP,\dot X\in M$ and there is some $q\leq p$ such that $q$ is an $M$-generic condition. Namely, every dense open set $D\in M$ is predense below $q$.
\end{definition}
Note that the model $M$ depends very much on the choice of $p$ and $\dot X$. So quasiproperness is still far from properness. We follow \cite{AsperoKaragila:2020} and define $H(\kappa)$ to be $\{x\mid\kappa\nleq^*|\tcl(\{x\})|\}$, but we can just as well work with $V_\alpha$ for a large enough limit ordinal $\alpha$ for all intents and purposes.

\begin{theorem}[$\ZF+\DC$]
Let $\PP$ be a forcing notion. If $\PP$ is quasiproper, then it preserves $\DC$. If $\PP$ is $\sigma$-sequential and preserves $\DC$, then $\PP$ is quasiproper.
\end{theorem}
\begin{proof}
  The core of the first part of the theorem is the same proof as Theorem~4.6 from \cite{AsperoKaragila:2020}, suppose that $\dot T$ is a name for a tree without maximal nodes, then for every $p$ there is some suitable model, $M$ and an $M$-generic $q\leq p$. Note that if $q$ is $M$-generic, then $q\forces``\dot T\cap M$ is a countable subtree of $\dot T$ without maximal nodes'', and so $q$ forces that $\dot T$ must have a branch. But the above just means that the set of conditions $q$ which are $M$-generic for some suitable $M$ is dense, which guarantees that $\dot T$ is forced to have a branch, and therefore $\DC$ is preserved.

  For the second part, suppose that $\PP$ is $\sigma$-sequential and that $\DC$ is preserved. Fix any $p,\dot X$ in $V$, fix a large enough $\kappa$ and consider the set $\cM$ of countable elementary submodels of $H(\kappa)$ which contain $p,\PP$ and $\dot X$. Working in $V[G]$, where $G$ is $V$-generic with $p\in G$, we define a relation on $\cM$: $M\sqsubset N$ if and only if
  \begin{enumerate}
  \item $N$ is an elementary extension of $M$,
  \item $G\cap N\cap D\neq\varnothing$ for every dense open $D\in M$
  \end{enumerate}

  We first show that if $M\in\cM$, then there is some $N\in\cM$ such that $M\sqsubset N$. Let $M$ be such model, then we can enumerate all the dense open sets in $M$ as $\{D_n\mid n<\omega\}$ and using $\DC$ there is a sequence of conditions $p_n\in G\cap D_n$ for all $n<\omega$. The sequence $\tup{p_n\mid n<\omega}$ lies in the ground model, since $\PP$ is $\sigma$-sequential. $\DC$ implies that there is an elementary submodel in $\cM$ generated by adding $\{p_n\mid n<\omega\}$ to $M$.

  Employing $\DC$ in $V[G]$, we have a sequence of models $\tup{M_n\mid n<\omega}$ such that $M_n\sqsubset M_{n+1}$ for all $n$. This sequence is again in $V$, and its union  $M=\bigcup M_n$, is a countable elementary submodel of $H(\kappa)$. In $V[G]$ we have that for every dense open $D\in M$, $D\cap M\cap G\neq\varnothing$: if $D\in M$, then $D\in M_n$ for some $n$, and therefore in $M_{n+1}$ there is a condition in $D\cap G\cap M_{n+1}$. Therefore there is some $q\leq p$ which is $M$-generic as wanted.
\end{proof}

\section{Main results}\label{sec:major}
\begin{theorem}\label{thm:main-dc}
  Let $\kappa$ be any infinite cardinal. It is consistent with $\ZF+\DC_{<\kappa}$ that
\begin{enumerate}
\item there is a $\kappa$-distributive forcing which violates $\DC$.
\item there is a $\kappa$-sequential forcing which violates $\AC_\omega$.
\end{enumerate}
\end{theorem}
\begin{proof}
  Let $\kappa$ be an uncountable regular cardinal and consider the $\kappa$-Cohen model, as described in \autoref{ex:cohen}. The case of $\kappa=\omega$ is vacuously true since $\AC_\omega$ already fails in the Cohen model. Denote by $M$ the symmetric extension, and by $\tup{\PP,\sG,\sF}$ the symmetric system. As usual, we will omit the dots from names to denote their interpretation in $M$.

  As explained in the above example, $M\models\DC_{<\kappa}$. We will describe two partial orders in this model which will witness the two failures. The first partial order will add a tree structure on $A$ which will witness the failure of $\DC$, the second will add an amorphous partition.\footnote{An infinite set is amorphous if all of its subsets are finite or co-finite; $\AC_\omega$ implies there are no amorphous sets.} In both cases, the idea is to consider the natural symmetric system which adds these objects ``directly'' and factor it into these two steps: first a symmetric extension adding a set of subsets of $\kappa$, then add the structure that would naturally be added by the ``direct'' symmetric extension.

  Working in $M$ let $\QQ_0$ be the partial order given by all the well-orderable and well-founded trees on $A$,\footnote{For well-orderable trees well-foundedness is equivalent to the inexistence of infinite branches.} ordered by $t_1\leq t_0$ if and only if $t_0$ is downwards closed in $t_1$. We claim that first of all, $\QQ_0$ is $\kappa$-distributive, and secondly if $H\subseteq\QQ_0$ is $M$-generic, $H$ defines a tree on $A$ which is of height $\omega$, without maximal nodes, and without branches, witnessing that $\DC$ fails in $M[H]$.

  First we note that if $t\in\QQ_0$, then $t$ has a canonical $\PP$-name in $\HS$. Since we are not adding any sets of size $<\kappa$ to $V$, therefore there is some $T\in V$ which is a well-founded tree on a bounded subset of $\kappa$, and $\dot t=\{\tup{\dot a_\alpha,\dot a_\beta}^\bullet\mid\tup{\alpha,\beta}\in T\}^\bullet$ is a $\PP$-name for $t$. We will use $T$ and $\dot t$ to correspond between this tree and the condition in $\QQ_0$, and because of this canonicity, there is no confusion when we treat them interchangeably where appropriate.

  Let $\gamma<\kappa$ and let $\tup{D_\alpha\mid\alpha<\gamma}\in M$ be a sequence of dense open subsets of $\QQ_0$. This sequence has a name in $\HS$, and since $\sF$ is $\kappa$-complete, we can simply choose names $\dot D_\alpha$ for each $\alpha<\gamma$ and consider $\tup{\dot D_\alpha\mid\alpha<\gamma}^\bullet$ as our canonical name.

  Let $p$ be a fixed condition in $\PP$ which forces that each $\dot D_\alpha$ is a dense open set, and let $\dot t$ be a canonical name for a condition. Fix $E$ such that $\supp p,\sym(\dot t)$, and for all $\alpha$, $\sym(\dot D_\alpha)$ all contain $\fix(E)$. We may even assume, without loss of generality that $\dom T=E$, else we can simply extend $\dot t$ as necessary.

  Let $p'\leq p$ be a condition such that for each $\alpha<\gamma$, there is some canonical $\dot t_\alpha$ such that $p'\forces\dot t_\alpha\in\dot D_\alpha$ and $\dot t\subseteq\dot t_\alpha$.\footnote{This is the same as requiring that $p'\forces\dot t_\alpha\leq_{\QQ_0}\dot t$ as both are canonical names.} Let $E'$ be a large enough set such that $\fix(E')\subseteq\sym(\dot t_\alpha)$ for all $\alpha<\gamma$, and $E\cup\supp p'\subseteq E'$. Such $E'$ exists since $\kappa$ is regular and $\gamma<\kappa$.

  For each $\alpha<\gamma$, pick $\pi_\alpha\colon\kappa\to\kappa$ to be a permutation such that $\pi_\alpha\in\fix(E)$, and letting $\pi_\alpha``(E'\setminus E)=E_\alpha$ we have that $\{E_\alpha\mid\alpha<\gamma\}$ are all pairwise disjoint. These exist since $|E'|<\kappa$. Observe the following:
  \begin{enumerate}
  \item $q=\bigcup_{\alpha<\gamma}\pi_\alpha p'$ is a condition, since $\dom\pi_\alpha p'\cap\dom\pi_\beta p'=E$.
  \item $\pi_\alpha p'\forces\pi_\alpha\dot t_\alpha\in\dot D_\alpha$ and $\dot t\subseteq\pi\dot t_\alpha$.
  \item If $\xi\in\dom\pi_\alpha T_\alpha\cap\dom\pi_\beta T_\beta$ for any $\alpha\neq\beta$, then $\xi\in E$.
  \end{enumerate}
  It follows from the three conditions that setting $\dot s=\bigcup_{\alpha<\gamma}\pi_\alpha\dot t_\alpha$ is a condition. If it were not a tree then any pair witnessing this must be already in $\dot t$ itself, by the third condition, which is impossible. Similarly, if $\dot s$ is not well-founded, then by the third condition it means some $\dot t_\alpha$ was not well-founded.

  But this means that $q\forces\dot s\in\dot D_\alpha$ for all $\alpha$. So given any $p$ and $\dot t$, we can extend $p$ to $q$ and find $\dot s$ such that $q\forces\dot t\subseteq\dot s\in\dot D_\alpha$ for all $\alpha$, and therefore the intersection of the $D_\alpha$ is dense.

  Next, it is easy to see that if $H\subseteq\QQ_0$ is $M$-generic, then $T=\bigcup H$ defines a tree structure on $A$. Standard density arguments show that this tree has height $\omega$ and no maximal elements. Finally, since $\QQ_0$ is $\sigma$-distributive, it adds no new $\omega$-sequence. So it is enough to show that if $\{a_n\mid n<\omega\}\subseteq A$ is in $M$, then it is not a branch in $T$. But this is again a simple density argument, given any condition $t$, pick any point in $t$, and whatever $a_n$s are not already mentioned in $t$, add as immediate successors of the chosen point. Therefore, by density argument no ground model set is a branch, and so $T$ is indeed without branches and serves as a counterexample to $\DC$.\footnote{By \autoref{thm:dist-pres}, $\AC_{<\kappa}$ is preserved in $M[H]$.}

  Indeed, this is the essence of the standard proof that $\AC_{<\kappa}$ does not imply $\DC$: first force with $\Add(\kappa,\kappa^{<\omega})$, take the automorphism group of the tree $\kappa^{<\omega}$ and generate the supports by fixing well-founded trees of rank $<\kappa$. See Theorem~8.12 in \cite{Jech:AC} for a similar construction in the context of permutation models.

  For the second partial order, let $\QQ_1$ be the partial order given by finite partitions of well-orderable subsets of $A$. Namely, a condition is a finite set, $e$, consisting of pairwise disjoint well-orderable subsets of $A$.

  We will denote $\bigcup e$ as $\dom e$, and given $a\in A$, we will write $e(a)$ to denote the cell containing $a$, which may be empty if $a\notin\dom e$. Given some $A'\subseteq A$, we will also write $e\restriction A'=\{C\cap A'\mid C\in e\}$.

  We define the order by $e_2\leq_{\QQ_1} e_1$ if and only if $e_2\restriction\dom e_1=e_1$. In other words, $e_2$ can extend the cells of $e_1$ or adds new ones, but it not merge any distinct cells.

  We need to show that $\QQ_1$ is $\kappa$-sequential and that if $H\subseteq\QQ_1$ is $M$-generic, then $\bigcup H$ is an amorphous partition of $A$. This will show that $M[H]\models\lnot\AC_\omega$, as wanted. Note that it is easy to see that $\QQ_1$ is not even $\sigma$-distributive by considering $D_n=\{e\in\QQ_1\mid n\leq|e|\}$.

  Note if $e\in\QQ_1$, then there is finite partition $E$ of some bounded subset of $\kappa$ such that $\dot e=\{\{\dot a_\alpha\mid\alpha\in C\}^\bullet\mid C\in E\}^\bullet$ is a name for $e$. We will adopt a similar convention to the previous case, that $E$ is the finite partition defining $\dot e$ and vice versa. And when it will be clear from context, we may also conflate $E$ and $\dot e$ to simplify the text, so if $S\subseteq\kappa$, the meaning of $\dot e\restriction S$ is clear: it is the condition corresponding to $E\restriction S$.

  One important consequence of the existence of these canonical names is that $\QQ_1$, as an ordered set, has a canonical name that is stable under all the automorphisms in $\sG$. This means that we can apply $\pi\in\sG$ to statements of the form $p\forces^\HS_\PP\dot e\forces_{\QQ_1}\varphi$ without having to worry that $\pi$ will somehow change the meaning of $\forces_{\QQ_1}$.

  Suppose that $\dot f\in M$ is a $\QQ_1$-name for a new $\gamma$-sequence of elements of $M$, for some $\gamma<\kappa$. We may assume, without loss of generality that every name appearing in $\dot f$ is of the form $\tup{\check\alpha,\check x}^\bullet$ for some $x\in M$. Let $[\dot f]\in\HS$ be a $\PP$-name for $\dot f$, for example, one such that every name that appears in it has the form $\tup{\dot e,\tup{\check\alpha,\dot x}^\bullet}^\bullet$, where $\dot e$ is some canonical name for a condition in $\QQ_1$ and $\dot x$ is a name in $\HS$ for the canonical $\QQ_1$-name, $\check x$, in $M$.

  Let $S\in[\kappa]^{<\kappa}$ such that $\fix(S)\subseteq\sym([\dot f])$. Let $p\in\PP$ be any condition such that $p\forces^\HS_\PP\dot e\forces_{\QQ_1}[\dot f](\check\alpha)=\dot x$ for some $\dot e$ and $\alpha<\gamma$.

  \begin{claim}
    $p\forces^\HS_\PP\dot e\restriction S\forces_{\QQ_1}[\dot f](\check\alpha)=\dot x$.
  \end{claim}
  \begin{proof}[Proof of Claim]\renewcommand{\qedsymbol}{$\square$ (Claim)}
    Suppose that $\dot e'$ is a name for a condition extending $\dot e\restriction S$. We can find $\pi\in\fix(S)$ such that $\pi p$ is compatible with $p$ and $\pi\dot e$ is compatible with both $\dot e$ and $\dot e'$ by mapping $\dom\dot e\setminus S$ and $\supp p\setminus S$ ``far enough'' from $\dom\dot e'$ and $\supp p$. Then we have that \[\pi p\forces^\HS_\PP\pi\dot e\forces_{\QQ_1}[\dot f](\check\alpha)=\pi\dot x.\] Since $\pi p$ and $p$ are compatible we can set $q=p\cup\pi p$ and get that \[q\forces^\HS_\PP\dot e\forces_{\QQ_1}[\dot f](\check\alpha)=\dot x\land\pi\dot e\forces_{\QQ_1}[\dot f](\check\alpha)=\pi\dot x.\]
    But since $\dot e$ and $\pi\dot e$ are compatible, it must be that $q\forces^\HS_\PP\dot x=\pi\dot x$, and since $\pi\dot e$ is compatible with $\dot e'$, it must be that $\dot e'$, if it decides the value of $\dot f(\check\alpha)$ at all, decides the same value.
  \end{proof}

  It follows that in $M$ a condition whose domain includes $S$ must have decided all the values of $\dot f$, and therefore it is a $\QQ_1$-name for a sequence already in $M$.

  Finally, we need to prove that the generic partition is amorphous. Suppose that this is not the case and let $\dot B$ be a $\QQ_1$-name in $M$ for an infinite co-infinite set of equivalence classes, and as before denote by $[\dot B]$ a $\PP$-name in $\HS$ for $\dot B$. Let $S\in[\kappa]^{<\kappa}$ such that $\fix(S)\subseteq\sym([\dot B])$, and let $p$ and $\dot e$ be such that $\supp p=\dom E=S$ and $p\forces^\HS_\PP\dot e\forces_{\QQ_1}[\dot B]$ is infinite and co-infinite.

  Pick some $\alpha,\beta\notin S$, and extend $p$ and $\dot e$ to $p'$ and $\dot e'$ such that:
  \begin{enumerate}
  \item $p'\forces^\HS_\PP\dot e'\forces_{\QQ_1}\dot e'(\dot a_\alpha)\in[\dot B]$ and $\dot e'(\dot a_\beta)\notin[\dot B]$.
  \item $\alpha$ and $\beta$ are added to new cells in $E'$, as opposed to cells that already exist in $E$.
  \item The cardinality of the cells of $\alpha$ and $\beta$ in $E'$ is equal.
  \item $p'\restriction E'(\alpha)$ and $p'\restriction E'(\beta)$ have the same type, in other words, they can be switched by some $\pi\in\sG$.
  \end{enumerate}

  This can be done by first finding extensions so that (1)--(3) are satisfied, then in $V$ we just add more elements to the cells of $\alpha$ and $\beta$ to ensure that we can find $p'$ as in the (4).

  Note that switching the two cells, of $\alpha$ and $\beta$, in $E'$ can be done, if at all, without moving any point in $S$. Picking such automorphism, $\pi$, we get that $\pi p'=p'$ and $\pi\dot e'=\dot e'$, and by $\pi\in\fix(S)$ we also get that $\pi[\dot B]=[\dot B]$. This is an outright contradiction, since applied to (1) the roles of $\dot a_\alpha$ and $\dot a_\beta$ are switched.
\end{proof}
We point out that the proof that $\QQ_2$ adds an amorphous set is based on the proof of Theorem~4.5 in \cite{Monro:1983}, where Monro uses a similar argument over the Cohen model, i.e.\ the case where $\kappa=\omega$, to add an amorphous set.
\begin{corollary}
$\ZF+\DC_{<\kappa}$ cannot prove that a $\sigma$-sequential forcing is $\sigma$-distributive for any uncountable $\kappa$.\qed
\end{corollary}
\section{Open problems}\label{sec:questions}
We saw that $\ZF$ cannot prove that a $\sigma$-sequential forcing is $\sigma$-distributive; but we also saw that assuming the consistency of suitable large cardinal axioms, the equivalence of $\sigma$-sequential to $\sigma$-distributive does not imply the axiom of choice either. Indeed, we can replace that $\sigma$ by any ${\leq}|X|$.

\begin{question}
  What is the consistency strength of $\ZF+\lnot\AC+$``$\sigma$-sequential forcing is $\sigma$-distributive''? Is it any different to $\forall X({\leq}|X|\text{-sequential}\to{\leq}|X|\text{-distributive})$?
\end{question}

In \cite{KaragilaSchlicht:Count} the first author proved with Philipp Schlicht that if $A$ is an infinite set such that $[A]^{<\omega}$ is Dedekind-finite,\footnote{Namely, there is no countable set of finite subsets of $A$.} then the forcing $\Add(A,1)$ given by finite partial functions $p\colon A\to 2$, which is clearly not $\sigma$-distributive,\footnote{Consider the sequence defined by $D_n=\{p\in\Add(A,1)\mid n<|\dom p|\}$.} satisfies that every antichain is finite, and equivalently ``every forcing statement is decided by a finite set''.\footnote{If $\varphi(\dot x)$ is a formula, then $\{p\in\Add(A,1)\mid p\text{ is }\subseteq\text{-minimal and } p\forces\varphi(\dot x)\}$ is finite.} These are conditions (2) and (4) in Theorem~6.1 in the paper.

\begin{claim}
  Suppose that $A$ is an infinite set such that $[A]^{<\omega}$ is Dedekind-finite, then $\Add(A,1)$ is $\sigma$-sequential.
\end{claim}
\begin{proof}
  Let $\dot f$ be a name such that $\1\forces\dot f\colon\check\omega\to\check V_\alpha$ for some $\alpha$.

  Consider for each $x\in V_\alpha$ the sets $M_n^x$ of maximal conditions $p\in\Add(A,1)$ which force $\dot f\restriction\check n=\check x$. This set is finite, so the set $X=\bigcup\{\dom p\mid p\in M_n^x\}$ is a finite set. For each $x$, consider now the finite antichain $A_x$, \[\{p\in\Add(A,1)\mid\dom p=X\land p\forces\dot f\restriction\check n=\check x\}.\]
  For any possible $x$ where $M_n^x$ is not empty to begin with, $A_x$ is a uniformly defined antichain, and moreover, if $x\neq y$, then $A_x\cup A_y$ is an antichain. Therefore $F_n$, defined as $\bigcup\{A_x\mid M_n^x\neq\varnothing\}$, is an antichain as well, and therefore finite.

  Finally, consider now the sequence of finite sets given by $\bigcup\{\dom p\mid p\in F_n\}$. Note that this sequence is increasing, since a condition in $F_{n+1}$ must extend some condition in $F_n$. It follows that the sequence is eventually constant, with some value $A'$ and therefore if $A'\subseteq\dom p$, then $p$ must decide $\dot f\restriction\check n$ for all $n<\omega$, which is to say that $p$ forces that $\dot f$ is in the ground model.
\end{proof}
Easily the proof above extends to any $\kappa$, so $\COrd\subseteq\cS_{\Add(A,1)}$, the sequentiality spectrum of $\Add(A,1)$. This extends Theorem~6.1 in \cite{KaragilaSchlicht:Count}, in which a list of properties of $\Add(A,1)$ which are all equivalent to $[A]^{<\omega}$ being Dedekind-finite are given. Specifically, conditions (8) and (9) which state that no reals and no sets of ordinals are added.

\begin{question}
  Suppose that every $\sigma$-sequential forcing is $\sigma$-distributive. The above claim show that $[A]^{<\omega}$ is Dedekind-infinite for any infinite set. Can we say more?
\end{question}

We finish this paper with two slightly orthogonal questions about the Foreman Maximality Principle (see \cite{FMS:1986}) which states that every nontrivial forcing adds a real or collapses cardinals. The consistency of this principle with $\ZFC$ is still open, but it is known to imply the consistency of large cardinals, as it implies that $\GCH$ fails everywhere. We saw that in the Gitik model every nontrivial forcing must add a countable sequence of ground model elements. But we can show that not every forcing adds a real, e.g.\ by showing that some of the collapsing sequences that are removed from the model by symmetric arguments are generic over it, and adding them back will not add new reals.

\begin{question}
  Does the Foreman Maximality Principle hold in the Gitik model?
\end{question}
\begin{question}
  What happens when we consider ``collapse cardinals'' in its general sense, meaning we add a bijection between two sets that did not have a bijection between them in the ground model. Does this modified principle hold in the Gitik model? Can it hold in $\ZF$ without large cardinals?
\end{question}
\bibliographystyle{amsplain}
\providecommand{\bysame}{\leavevmode\hbox to3em{\hrulefill}\thinspace}
\providecommand{\MR}{\relax\ifhmode\unskip\space\fi MR }
\providecommand{\MRhref}[2]{%
  \href{http://www.ams.org/mathscinet-getitem?mr=#1}{#2}
}
\providecommand{\href}[2]{#2}

\end{document}